\documentclass[reqno]{amsart}
\usepackage{amssymb,amsmath,amsfonts,amsthm}
\usepackage{graphics,ifthen,mathrsfs,verbatim,epsfig,MnSymbol}
\usepackage[mathcal]{euscript}
\usepackage{enumitem}
\usepackage{color}
\usepackage{hyperref,complexity}
\hypersetup{colorlinks=true,linkcolor=green,pdfborderstyle={/S/U/W 1}}

\theoremstyle{plain}
\newtheorem{theorem} {Theorem} [section]
\newtheorem{lemma} [theorem] {Lemma}
\newtheorem{corollary}[theorem] {Corollary}

\newtheorem{fact} [theorem] {Fact}
\newtheorem{proposition} [theorem] {Proposition}
\newtheorem{claim} [theorem] {Claim}
\newtheorem*{claim*} {Claim}
\newtheorem{conjecture}[theorem] {Conjecture}

\newtheorem{question}[theorem] {Question}

\theoremstyle{definition}

\newtheorem{definition}[theorem] {Definition}

\let\n\noindent%

\newcommand{\iso}[1][]{\buildrel {#1} \over \cong}

\usepackage{pgffor}
\foreach \x in {A,...,Z}{%
\expandafter\xdef\csname b\x\endcsname{\noexpand\ensuremath{\noexpand\mathbf{\x}}}
\expandafter\xdef\csname bb\x\endcsname{\noexpand\ensuremath{\noexpand\mathbb{\x}}}
\expandafter\xdef\csname c\x\endcsname{\noexpand\ensuremath{\noexpand\mathcal{\x}}}
\expandafter\xdef\csname s\x\endcsname{\noexpand\ensuremath{\noexpand\mathscr{\x}}}
}

\def\H{\ifmmode{\mathbb H}\else\accent"07D\fi} 

\renewcommand{\PH}{\cS_{H}}
\newcommand{\unary}{{c}}
\newcommand{\br}{\mathrm{br}}
\newcommand{\pr}{\mathrm{pr}}

\newcommand{\redSub}[1]{#1^R}
\newcommand{\blueSub}[1]{#1^B}
\newcommand{\pleq}{\leq_{\rm poly}}
\newcommand{\peq}{=_{\rm poly}}

\newclass{\Hom}{Hom}
\newclass{\lHom}{LHom}
\newclass{\lsHom}{LSwHom}
\newclass{\sHom}{SwHom}
\newclass{\NAESAT}{NaeSAT}
\newcommand\lpHom[1]{\lHom(\cS_{#1})}
\newclass{\NU}{NU}
\mathchardef\mhyphen="2D
\newcommand{\kNU}[1]{#1\mhyphen\NU}
\newclass{\WNU}{WNU}

\newcommand{\fp}{f_p}
\newcommand{\ps}{p_s}
\newcommand{\sm}{f_s}
\newclass{\CNU}{CNU}
\newclass{\CSP}{CSP}

\DeclareMathOperator{\maj}{maj}
\DeclareMathOperator{\mCP}{CP}
\newcommand{\CPC}[1]{\mCP(\Bip{H}^{#1})}
\DeclareMathOperator{\rc}{rc}
\newcommand{\Bip}[1]{\cB_{#1}}

\usepackage{tikz}
\usetikzlibrary{positioning}
\usetikzlibrary{snakes}

\tikzset{
 blackvertex/.style={circle, draw=black!100,fill=black!100,thick, inner sep=0pt, minimum size=2mm},
 cvert/.style={circle, draw=black!100,fill=none,thick, inner sep=2pt, minimum size=2mm},
 empty/.style={draw=none, fill=none}
}

\tikzset{decoration={snake,amplitude=.4mm,segment length=2mm,
 post length=0mm,pre length=0mm}} 

 \newcommand{\Verts}[2][v]{\foreach \i/\x/\y in {#2}{\draw (\x,\y) node (#1\i){};}}
 
 \newcommand{\Edges}[3][black]{\foreach \i/\j in {#2}{\draw[#1] (\i) edge[#3] (\j);}}
 
 \newcommand{\Vlabel}[4][.3cm]{\draw node[empty, #3 of = #2, node distance = #1] () {$#4$};}
 \newcommand{\Vlabels}[1]{\foreach \lab/\pos/\vert in {#1}{\Vlabel{\vert}{\pos}{\lab}}}

\begin{document}
 \thispagestyle{empty}
 \title[Towards a dichotomy for list switch homomorphism]{Towards a dichotomy for the list switch homomorphism problem for signed graphs}
 \author{Hyobin Kim, Mark Siggers}
 \address{Kyungpook National University Mathematics Department,
80 Dae-hak-ro, Daegu Buk-gu, South Korea, 41566}
 \email{hbkim1029@knu.ac.kr, mhsiggers@knu.ac.kr}
\thanks{The first author is supported by the KNU BK21 project. The second author is supported by Korean NRF Basic Science Research Program (2018-R1D1A1A09083741) funded by the Korean government (MEST), and the Kyungpook National University Research Fund}
\begin{abstract}
 We make advances towards a structural characterisation of the signed graphs $H$ for which the list switch $H$-colouring problem $\operatorname{LSwHom}(H)$ can be solved in polynomial time.  We conjecture a characterisation in the case that the graph $H$ can be switched to a graph in which every negative edge is also positive, and prove the characterisation in the case that the signed graph is reflexive. 
\end{abstract}
 \maketitle

 \section{Introduction}\label{sect:intro} 

The $\CSP$-dichotomy of Bulatov \cite{Bu17} and Zhuk \cite{Zh17} tells us that the constraint satisfaction problem $\CSP(\cH)$ for a core relational structure $\cH$ is polynomial time solvable, or in $\P$, if $\cH$ admits a $\WNU$-polymorphism, and is otherwise $\NP$-complete, or in $\NPC$. It is difficult, however, to decide if $\cH$ admits a $\WNU$-polymorphism. So more tractable dichotomies, characterisations like the well known $H$-colouring dichotomy of \cite{HN90} which says that the homomorphism problem $\Hom(H)$ for a simple graph $H$ is in $\P$ if and only if $H$ is bipartite, are still sought. One such dichotomy arose recently in relation to signed graphs \cite{BFHN,BS}, and in this paper we work towards a generalisation of this. The definition of a $\WNU$-polymorphisms, as well as the definitions of many other now standard algebraic terms introduced in this section without definition, can be found in Section \ref{sect:back}. 

 Henceforth, {\em graphs} are undirected graphs in which loops-- edges of the form $\{v,v\}$-- are allowed. If all vertices have loops then the graph is {\em reflexive}, and if no vertices have loops then it is {\em irreflexive}. We denote an edge $\{u,v\}$ of a graph simply as $uv$, and write $u\sim v$ to mean that $uv$ is an edge.

 A {\em signed graph} is a graph $G$ together with an assignment of a sign $+$ or $-$ to each edge. Introduced by Harary in \cite{Harary} in 1955, there are numerous results about signed graphs. For our purposes, it will be enough to view a signed graph $G$ simply as a {\em $\br$-graph} (for blue-red): a pair of graphs-- $\blueSub{G}$, whose edges are called {\em blue edges}, and $\redSub{G}$, whose edges are {\em red edges}-- on the same vertex set. Signed graphs, as defined by Harary, were irreflexive and simple. But as is done in \cite{BFHN,BS} we allow loops, and allow $\blueSub{G} \cap \redSub{G}$ to be non-empty. An edge in $\blueSub{G} \cap \redSub{G}$ is technically two edges in $G$, one of each colour, but it is convenient to consider it a single edge with both colours. As such, we refer to such an edge as a {\em purple edge}. If all edges of a $\br$-graph $G$ are red (or purple) we call $G$ {\em red} (or {\em purple}). If an edge is red (or blue) but not purple, we call it {\em pure} red (or pure blue). A $\br$-graph $G$ is irreflexive, or bipartite, or connected or such, if the {\em underlying graph} $\redSub{G} \cup \blueSub{G}$ is. We write $u \sim v$ to mean that $uv$ is in $\redSub{G} \cup \blueSub{G}$. Many results about signed graphs deal with the operation of {\em switching}, introduced by Zaslavsky~\cite{Z82b}, in which, for a vertex $v$, one switches, with respect to $\redSub{G}$ and $\blueSub{G}$, the set of edges incident to $v$. As a loop at $v$ can be considered to be switched twice, loops are unchanged by a switching. Purple edges, being a red and a blue edge, can also be viewed as unchanged by a switching. Two $\br$-graphs are {\em switching-equivalent} if one can be changed to the other by a series of switchings. A $\br$-graph that is switching-equivalent to a graph with no pure blue edges is a {\em $\pr$-graph} (for purple-red). 
 
As a {\em homomorphism} $\phi: G \to H$ from a graph $G$ to a graph $H$ is a edge preserving vertex map, a {\em homomorphism} from a $\br$-graph $G$ to another $\br$-graph $H$ is a vertex map $\phi: G \to H$ which preserves both red edges and blue edges. Defined in \cite{NRS14}, a {\em switch-homomorphism} from $G$ to $H$ is a vertex map $\phi: G \to H$ such that for some graph $G'$ that is switching-equivalent to $G$, the map $\phi: G' \to H$ is a homomorphism. 

 In \cite{FN14}, Foucaud and Naserasr introduced the decision problem $\sHom(H)$ for a given $\br$-graph $H$: decide for a given $\br$-graph $G$ if there is a switch-homomorphism to $H$. They addressed the problem of classifying the complexity of $\sHom(H)$ in terms of $H$. A priori, $\sHom(H)$ is not a $\CSP$ problem, but in \cite{BFHN} the authors showed that it is polynomially equivalent to $\CSP(\PH)$ for a $\br$-graph $\PH$ called the {\em switching graph} of $H$, defined in Definition \ref{def:switchgraph}. From this, it follows by the CSP-dichotomy, that there is a complexity dichotomy. This 
 dichotomy was characterised in \cite{BFHN} and \cite{BS}.

The {\em switch-core} of a $\br$-graph $H$ is the unique minimal induced subgraph to which it admits a switch-homomorphism. A complexity dichotomy for the switch-homomorphism problem was conjectured in \cite{BFHN}, and completed in \cite{BS}. 

\begin{theorem}[\cite{BS}]\label{thm:switchhom} For a $\br$-graph $H$, the problem $\sHom(H)$ is in $\P$ if the switch-core of $H$ has at most two edges, and is otherwise in $\NPC$.
\end{theorem} 

Our goal is to determine, for a given $\br$-graph $H$, the complexity of the list version $\lsHom(H)$ of the problem $\sHom(H)$. 

Much has already been done in \cite{BBFHJ}. Building on known ideas that we explain in more detail in the next section, the authors first observe that $\lsHom(H)$ is in $\P$ if and only if $\PH$ admits a {\em semi-conservative} $\WNU$-polymorphism, which is defined at the end of Section \ref{sect:back}. 
The goal then becomes to characterise the $\br$-graphs $H$ such that $\PH$ admits a semi-conservative $\WNU$-polymorphism. 
The authors then go on to define a special bipartite-min ordering of a bipartite resolution $\Bip{H}$ of $H$, and two obstructions to its existence-- invertible pairs and chains in $\Bip{H}$, both of which are detectable in polynomial time. Using these as their main tools, they characterise the $\br$-trees admitting semi-conservative $\WNU$-polymorphisms. 
The full characterisation is not easy to state, but in the restricted case that $H$ is an irreflexive tree, it has a simpler statement. 

It is a simple task to show that an irreflexive $\br$-tree is a $\pr$-graph: working from a root outwards one can switch all pure blue edges to pure red. 
In \cite{BBFHJ}, the authors prove the following in the case that the $\pr$-graph $H$ is a tree. (The definition of $\Bip{H}$ is given in the beginning of Section \ref{sect:bipart}, and definitions of invertible pairs, chains, and special bipartite-min orderings of $\pr$-graphs are given at the beginning of Section \ref{sect:pr}.) 

\begin{conjecture}\label{conj:main}
 For a $\pr$-graph $H$, the following are equivalent.
 \begin{enumerate}
 \item $\Bip{H}$ has no invertible pairs and no chains.
 \item $\Bip{H}$ has a special bipartite-min ordering.
 \item $\PH$ has a semi-conservative $\WNU$-polymorphism. 
 \end{enumerate}
\end{conjecture}

Our main result is to prove this in the case that $H$ is a reflexive $\pr$-graph. Though irreflexive trees are $\pr$-graphs, reflexive trees need not be. So this does not strictly generalise even the reflexive results of \cite{BBFHJ}. That said, $\br$-trees are almost $\pr$-graphs in that the only pure blue edges that one cannot get rid of by switching are blue loops. 
These blue loops add considerable complication to the proofs of \cite{BBFHJ}, and it is evident from examples in \cite{BBFHJ} that Conjecture \ref{conj:main} will not generalise to $\br$-graphs without some alternation. 

Though we do not generalise the results of \cite{BBFHJ}, our result seems to be a significant complement to theirs, not only in the classes of graphs dealt with, but also in the techniques used. Our proofs make significant use of symmetries both in the definition of $\PH$, and in the definition of bipartite-min orderings.

 Section \ref{sect:back} can be viewed as, perhaps, an expansion of this introductory section. We recall many required algebraic results of $\CSP$ theory and structural results about the list homomorphism problem for graphs. 
 
 In Section \ref{sect:br} we define a switch-symmetric polymorphism which commutes with the obvious symmetry in $\PH$, and use it to observe that if a polymorphism of the red subgraph $\redSub{\PH}$ of $\PH$ is switch-symmetric, then it is a polymorphism of $\PH$. This is a useful tool in defining polymorphisms on $\PH$, but also yields a quick result. It allows us to quickly prove Theorem \ref{thm:main} which characterises the $\br$-graphs $H$ such that $\PH$ admits a conservative $\WNU$-polymorphism. 
 This depends on known characterisations of graphs with conservative $\WNU$-polymorphisms. As these do not exist for semi-conservative $\WNU$-polymorphisms though, much work remains. 
 
 In Section \ref{sect:bipart} we make observations about a `reduction to the bipartite case' for list homomorphisms of graphs that is given in \cite{FHH},
 and show that parts of it extend transparently to $\br$-graphs. 
 In particular, this leads to the definition of a `parity-symmetric' bipartite-min 
 ordering, and the following strengthening of Conjecture \ref{conj:main}.

\begin{conjecture}\label{conj:main2}
 For a $\pr$-graph $H$, the following are equivalent.
 \begin{enumerate}
 \item $\Bip{H}$ has no invertible pairs or chains.
 \item[(ii*)] $\Bip{H}$ has a parity-symmetric special bipartite-min ordering.
 \item[(iii)] $\PH$ has a semi-conservative $\WNU$-polymorphism. 
 \end{enumerate}
\end{conjecture}
 
 As $(ii*) \Rightarrow (ii)$ is trivial and $(ii) \Rightarrow (i)$ is shown in \cite{BBFHJ} this is indeed a stronger version of Conjecture \ref{conj:main}.
 In fact it consists of Conjecture \ref{conj:main} and the following, the proof of which, for graphs, follows from non-trivial results in \cite{FHH}. 
 
 \begin{conjecture}\label{conj:ps}
 For a $\pr$-graph $H$, $\Bip{H}$ has a special bipartite-min ordering if and only if it has a parity-symmetric one. 
 \end{conjecture}

 In Section \ref{sect:pr} we prove Conjecture \ref{conj:main2} for reflexive $\pr$-graphs. This is Theorem \ref{thm:Main}. This also proves Conjectures \ref{conj:main} and \ref{conj:ps} for reflexive $\pr$-graphs.  The main part proof of this consists of the reduction to the bipartite case from Section \ref{sect:bipart} and Theorem \ref{thm:specbmo} which shows that the implication $(ii^*) \Rightarrow (iii)$ holds for bipartite $\pr$-graphs. Together this gives us the implication $(ii^*) \Rightarrow (iii)$ for reflexive $\pr$-graphs. The other implications come from known results. 
 \footnote{The proof of $(i) \Rightarrow (ii*)$ for reflexive $\pr$-graphs was a big part of earlier versions of this paper.  We are indebted to a reviewer for finding mistakes in this proof. In trying to fix these mistakes, the proof grew considerably, and devolved into extensive casework. We have forgone this, as at about the same time a much more elegant proof came out in \cite{BBHJR}, in which the authors also verified Conjecture \ref{conj:main} for bipartite $\pr$-graphs. We have replaced our incomplete proof of this implication, which can still be found on the ArXiv, with a reference to \cite{BBHJR}.}


 The literature about the list homomorphism problem that we recall in Section \ref{sect:back} suggests many `polymorphism collapses' that may be useful in finding a full characterisation of the complexity of $\lsHom(H)$. While some have proved useful, some have been red herrings. In our final section, Section \ref{sect:counter} we give examples showing that some of these collapses do not occur.

\section{Background} \label{sect:back} 

In this section we recall the definitions required to properly state our problem, and make some initial observations about it. Many of these observations can also be found in \cite{BBFHJ}.

\subsection{Structures, CSP, and List Homomorphism}

A (relational) structure $\cH$ consists of a set $V = V(\cH)$ of vertices with a finite ordered set of finite arity relations $R_i \subset V^k$ on $V$. The corresponding ordered set of the arities of the structure is called its {\em type}. A {\em homomorphism} between two structures of the same type is a vertex map that preserves each relation. The problem $\CSP(\cH)$ for a structure $\cH$ is the problem of deciding if an instance structure $\cG$ of the same type admits a homomorphism to $\cH$.

A graph $H$ is viewed as a structure with one symmetric binary relation $E$, an edge $uv$ corresponding to the pairs $(u,v)$ and $(v,u)$ in $E$. A vertex map $V(G) \to V(H)$ of graphs is a homomorphism between two graphs if and only if it is a homomorphism of the corresponding symmetric binary structures. So $\Hom(H)$ is polynomially equivalent to $\CSP(H)$.

In the list variation $\lHom(H)$ of the homomorphism problem $\Hom(H)$, an instance $G$ comes attached with a list $L(v) \subset V(H)$ for each vertex $v \in V(G)$, and one must decide if there is a homomorphism $\phi:G \to H$ such that $\phi(v) \in L(v)$ for each vertex $v$ of $G$.
This is equivalent to $\CSP(H_\unary)$ where $\cH_\unary$ is the relational structure one gets from a structure $\cH$ by adding a new unary relation $L_S= \{ (s) \mid s \in S \}$ for each subset $S \subset V(\cH)$.

The {\em core} $\cC(\cH)$ of a structure $\cH$ is the unique minimum induced substructure to which it admits a homomorphism. It is basic that $\CSP(\cC(\cH))$ is polynomially equivalent to $\CSP(\cH)$.

 A $k$-ary polymorphism $\phi$ of $\cH$ is a homomorphism $\phi: \cH^k \to \cH$ (that is, a map preserving all relations of $\cH$), where $\cH^k$ is the $k$-time categorical product of $\cH$ with itself. It is {\em idempotent} if $\phi(h,h, \dots, h) = h$ for all $h \in V(\cH)$, and it is {\em conservative} if
$\phi(v_1, \dots, v_k) \in \{v_1, \dots, v_k\}$. An idempotent polymorphism $\phi: \cH^k \to \cH$, for $k \geq 3$ is {\em weak near unanimity} or $\WNU$ if for every choice of $x,y \in V(\cH)$ we have
 \[ \phi(x,x, \dots, x,y) = \phi(x,x, \dots, y,x) = \dots \phi(y,x, \dots, x,x). \]

\begin{theorem}[The $\CSP$-dichotomy, \cite{Bu17,Zh17}]\label{thm:CSPdichot}
 For a relational structure $\cH$ the problem $\CSP(\cH)$ is in $\P$ if the core of $\cH$ admits a $\WNU$-polymorphism, and is otherwise in $\NPC$.
\end{theorem}

 In general it is difficult to find the core of a structure $\cH$, 
 and tends to be difficult to give a non-algebraic description of the structures $\cH$ for which $\CSP(\cH)$ is in $\P$. For list colouring, it gets a little easier. For one, the structure $H_\unary$ is always a core. It is easy to see that any polymorphism of $H_\unary$ is conservative, so for list colouring problems we get the following corollary of Theorem \ref{thm:CSPdichot} which Bulatov actually proved years earlier in \cite{BulList}. 

\begin{corollary}
For a relational structure $\cH$ the problem $\CSP(\cH_\unary)$ is in $\P$ if and only if the core of $\cH$ admits a conservative $\WNU$-polymorphism.
\end{corollary}

\subsection{The switching graph}

One might observe that $\sHom(H)$ for a $\br$-graph $H$ is not a $\CSP$, so the above tools do not apply, but one of the main tools in \cite{BFHN} and \cite{BS} is the following construction which allows us to view $\sHom(H)$ as a $\CSP$.

For a product $A \times B$ of sets or of graphs, we use $\pi_A$ to denote projection map onto $A$ and $\pi_B$ the projection map onto $B$.
 
\begin{definition}\label{def:switchgraph}
 Let $H$ be a $\br$-graph, and $S$ be the set $\{0,1\}$. The {\em switching graph} $\PH$ of $H$ is the $\br$-graph on the set $V(H) \times S$, such that each of the sets 
 $\pi_S^{-1}(i) = \{ (v,i) \mid v \in v(H) \}$ induces a copy of $H$ (identified through the projection $\pi_H$ in the obvious way), and such that, where the {\em switch} map $\sm: V(\PH) \to V(\PH)$ is defined $\sm( (v,i)) = (v, 1-i)$, the edge $u\sm(v)$ is red if $uv$ is blue,
$uf_s(f)$ is blue if $uv$ is red. Again, the edge $u\sm(v)$ will be purple if $uv$ is purple. 
\end{definition}

\begin{figure}
\begin{center}
\begin{tikzpicture}[every loop/.style={},scale=0.7]
 \node[blackvertex,label={270:$u$}] (u) at (0,0) {};
 \node[blackvertex] (v) at (0,1.5) {};
 \node[blackvertex] (w) at (0,3) {};
 \node[blackvertex,label={90:$v$}] (x) at (0,4.5) {};
 \draw[thick,blue] (u) to[bend left=5] (v) (w)--(x);
 \draw[thick,red,dashed] (u) to[bend left=20] (x) (v)--(w) (u) to[bend right=5] (v);
 \path[thick,red,dashed] (x) edge[out=150,in=210,loop, min distance=10mm] node {} (x);
 \path[thick,blue] (u) edge[out=150,in=210,loop, min distance=10mm] node {} (u);

 \node at (0,-1) {$H$};

\begin{scope}[xshift=3.5cm]
 \node[blackvertex,label={270:$(u,0)$}] (u0) at (0,0) {};
 \node[blackvertex] (v0) at (0,1.5) {};
 \node[blackvertex] (w0) at (0,3) {};
 \node[blackvertex,label={90:$(v,0)$}] (x0) at (0,4.5) {};
 \draw[thick,blue] (u0) to[bend left=5] (v0) (w0)--(x0);
 \draw[thick,red,dashed] (u0) to[bend left=20] (x0) (v0)--(w0) (u0) to[bend right=5] (v0) ;

 \node[blackvertex,label={270:$(u,1)$}] (u1) at (2,0) {};
 \node[blackvertex] (v1) at (2,1.5) {};
 \node[blackvertex] (w1) at (2,3) {};
 \node[blackvertex,label={90:$(v,1)$}] (x1) at (2,4.5) {};
 \draw[thick,blue] (u1) to[bend left=5] (v1) (w1)--(x1);
 \draw[thick,red,dashed] (u1) to[bend right=20] (x1) (v1)--(w1) (u1) to[bend right=5] (v1);

 \path[thick,red,dashed] (x0) edge[out=150,in=210,loop, min distance=10mm] node {} (x0);
 \path[thick,blue] (u0) edge[out=150,in=210,loop, min distance=10mm] node {} (u0);

 \path[thick,red,dashed] (x1) edge[out=30,in=-30,loop, min distance=10mm] node {} (x1);
 \path[thick,blue] (u1) edge[out=30,in=-30,loop, min distance=10mm] node {} (u1);

 \draw[thick,blue] (x0)--(x1);
 \draw[thick,red,dashed] (u0)--(u1);

 \draw[thick,red,dashed] (u0) to[bend left=5] (v1) (w0)--(x1) (u1) to[bend left=5] (v0) (w1)--(x0);
 \draw[thick,blue] (u0) to[bend right=5] (v1) (u1) to[bend right=5] (v0) (u0)--(x1) (u1)--(x0) (v1)--(w0) (v0)--(w1);

 \node at (1,-1) {$\PH$};
\end{scope}

\end{tikzpicture}
\end{center}
\caption{A $\br$-graph $H$ and its switching graph $\PH$.}
\label{fig1}
\end{figure}
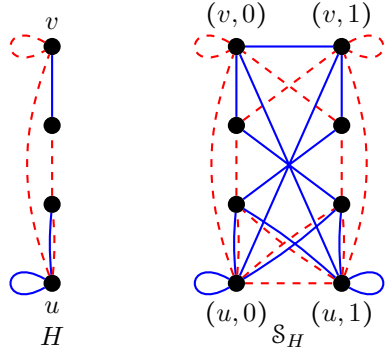

See Figure~\ref{fig1} for an example of the switching graph construction. Blue edges are solid, while red edges are dashed. 

The following proposition implies that $\sHom(H)$ is polynomially equivalent to
the homomorphism problem $\Hom(\PH)$ for the $\br$-graph $\PH$.
\begin{proposition}[\cite{BFHN}]\label{prop:switch}
 For $\br$-graphs $G$ and $H$, $G$ admits a switch-homomorphism to $H$ if and only if
 it admits a homomorphism to $\PH$.
\end{proposition}

\subsection{Definition of List Switch-Homomorphism}

Though the problems $\sHom(H)$ and $\Hom(\PH)$ are polynomially equivalent, they have different list versions. For a $\br$-graph $H$ we define the following two problems.

\medskip
\n $\lsHom(H)$ \\
Instance: A $\br$-graph $G$ with lists $L:V(G) \to 2^{V(H)}$.\\
Question: Does $G$ admit a switch-homomorphism $\phi:G \to H$ preserving lists?
\medskip

\medskip
\n $\lpHom{H}$ \\
Instance: A $\br$-graph $G$ with lists $L:V(G) \to 2^{V(\PH)}$.\\
Question: Does $G$ admit a homomorphism $\phi: G \to \PH$ preserving lists?
\medskip

The second problem is clearly equivalent to the $\CSP$ problem $\CSP(\PH^{\unary})$. The first problem can also be made equivalent to a $\CSP$ using Proposition \ref{prop:switch}. Indeed, from the proof of this proposition, it is not hard to see that it is polynomially equivalent to $\CSP(\PH^{s\unary})$ where we get $\PH^{s\unary}$ from $\PH^{\unary}$ by throwing away non-symmetric lists, i.e., throwing away the unary relation $L_S$ for any subset $S \subset V(H)$ that is not closed under action by the switch map $s$ from Definition \ref{def:switchgraph}.

Clearly $\lpHom{H}$ encodes $\lsHom(H)$ by taking only instances with symmetric lists, and as the full lists are symmetric, $\lsHom(H)$ encodes $\Hom(\PH)$ by taking instances with full lists on every vertex. Thus we have the following polynomial reductions. 
 \begin{equation}\label{eq:heir}
 \Hom(\PH) \pleq \lsHom(H) \pleq \lpHom{H}
 \end{equation}

 Calling a polymorphism of $\PH$ {\em semi-conservative} if it conserves symmetric lists, Theorem \ref{thm:CSPdichot} yields a nice algebraic description of the relationship between 
 the problems $\lsHom(H)$ and $\lpHom{H}$.

\begin{fact} \label{fact:WNU} 
 For a $\br$-graph $H$, $\lsHom(H)$ is in $\P$ if and only if $\PH$ admits a semi-conservative $\WNU$-polymorphism, and 
 $\lpHom{H}$ is in $\P$ if and only if $\PH$ admits a conservative $\WNU$-polymorphism.
\end{fact}
\begin{proof}
 Everything would be immediate from Theorem \ref{thm:CSPdichot} except that to apply it we need that the structures are cores, while the structure $\PH^{s\unary}$, which we use for $\lsHom(H)$, need not be a core. So we have to observe that if 
 $\PH$ admits a semi-conservative $\WNU$-polymorphism then the core of $\PH^{s\unary}$ admits a $\WNU$-polymorphism. To get the core of $\PH^{s\unary}$ we must identify the vertices $(v,0)$ and $(v,1)$ for any vertex $v$ of $H$ all of whose edges are purple. Let $r$ be the retraction of $\PH$ to its core that we get by making all of these identifications. It is easy to check that if $\phi: \PH^k \to \PH$ is a semi-conservative $\WNU$-polymorphism then $\phi' = r \circ \phi|_{r(\PH)^k}$ is a $\WNU$-polymorphism on the core $r(\PH)$. 
 \end{proof}

 \subsection{Collapses of Conservative Polymorphisms}\label{sub:collapses}
 
 Our goal is a nice characterisation of the $\br$-graphs $H$ such that $\PH$ admits a (semi-)conservative $\WNU$-polymorphism. In the case of other graph-like structures, such characterisations often come from the fact that a structure has a conservative $\WNU$-polymorphism if and only if it has some other nicer polymorphism such as a conservative $\kNU{3}$-polymorphism (also known as a conservative majority polymorphism). 

\begin{definition}
 A polymorphism $\phi: H^3 \to H$ on a graph $H$ is {\em $\kNU{3}$} if for any $x,y \in V(H)$, we have
 \[\phi(x,x,y) = \phi(x,y,x) = \phi(y,x,x) = x. \]
\end{definition}

 It is clear that a $\kNU{3}$-polymorphism is $\WNU$, but one can find many examples of structures with $\WNU$-polymorphisms that do not have $\kNU{3}$-polymorphisms.
 However, there is a collapse of these two types of polymorphisms for graphs. 
 
 \begin{theorem}[\cite{FHH}]\label{thm:FHHshort}
 A graph $H$ has a conservative $\WNU$-polymorphism if and only if it has a conservative $\kNU{3}$-polymorphism. 
\end{theorem}

 Another useful conservative $\WNU$-polymorphism is disguised as an ordering. 
 Recall that an ordering $\leq$ of a set $V$ can be viewed as a $2$-ary function $\min:V^2 \to V$ such that $\min(u,v) = \min(v,u) \in \{u,v\}$ for all $u,v \in V$.
 
\begin{definition}\label{def:min}
 An ordering $\leq$ of the vertices of a reflexive graph is a {\em min ordering} if the corresponding $2$-ary function $\min$ is a polymorphism. 
 \end{definition}
 A graph with a min ordering admits the induced $3$-ary operation $\min(a,b,c) = \min(\min(a,b), c)$ which is a conservative $\WNU$-polymorphism of $H$; we call this the {\em min polymorphism} associated with the min ordering. 
 For reflexive graphs, there is a collapse of $\WNU$-polymorphisms to min orderings.   
 \begin{theorem}[\cite{FederHellList}]\label{thm:min-order}
 For a reflexive graph $H$ the following are equivalent.
 \begin{itemize}
 \item The problem $\lHom(H)$ is in $\P$.
 \item $H$ admits a conservative $\WNU$-polymorphism.
 \item $H$ admits a min ordering. 
 \item $H$ contains no invertible pairs. 
 \end{itemize}
 \end{theorem} 
 
 The existence of min orderings is known to characterise the well known class of interval graphs, so a reflexive graph has a conservative $\WNU$-polymorphism if and only if it is an interval graph. 
 Min orderings are also known to be characterised by the fact that they satisfy the {\em underbar} property:
 \begin{equation}\label{eq:underbar}
 ( a \leq a', b \leq b', a \sim b', a' \sim b ) \Rightarrow ( a \sim b).
 \end{equation}
 
 No non-empty irreflexive graph can have a min ordering, but there is a generalisation, called a bipartite-min ordering, defined in Section \ref{sect:bipart} on what we call the bipartite resolution $\Bip{H}$ of $H$, that exists for many irreflexive graphs.

 It was shown in \cite{FV98}, for a graph $H$ with a bipartite-min ordering, that the problem $\lHom(H)$ can be solved in polynomial time via the arc-consistency algorithm. In \cite{HHMR} the authors generalised the notion of bipartite-min orderings to digraphs, unifying many known list homomorphism characterisations. 
 
 \begin{theorem}\label{thm:min-order-irr} \cite{HHMR}
 For a graph $H$ the following are equivalent.
 \begin{itemize}
 \item The problem $\lHom(H)$ is in $\P$.
 \item $H$ admits a conservative $\WNU$-polymorphism.
 \item $H$ admits a bipartite-min ordering. 
 \item $\Bip{H}$ contains no invertible pairs. 
 \end{itemize}
 \end{theorem}

\section{Dichotomy for List $\PH$-colouring} \label{sect:br}

Recall the switch automorphism $\sm: \PH \to \PH$ that maps $(v,i)$ to $(v,1-i)$. This is clearly also an automorphism of the red subgraph $\redSub{\PH}$ of $\PH$. A $k$-ary polymorphism $\phi$ of $\PH$, or of $\redSub{\PH}$, is {\em switch-symmetric} if it commutes with $\sm$:
\[ \phi(\sm(v_1), \dots, \sm(v_k)) = \sm(\phi(v_1, \dots, v_k)). \]

\begin{lemma}\label{lem:ss}
 If a polymorphism of $\redSub{\PH}$ is switch-symmetric, then it is a polymorphism of $\PH$.
\end{lemma}
\begin{proof}
 Let $\phi: (\redSub{\PH})^k \to \redSub{\PH}$ be switch-symmetric. To see that it is a polymorphism of $\PH$ it is enough to observe that it is also a polymorphism of $\blueSub{\PH}$. Let $x_i \sim y_i$ in $\blueSub{\PH}$ for each $i \in [k]$.
 By the definition of $\PH$ we have that $x_i \sim \sm(y_i)$ in $\redSub{\PH}$ for each $i$. So as $\phi$ is a polymorphism of $\redSub{\PH}$
 we get
 \[ \phi(x_1, \dots, x_k) \sim \phi(\sm(y_1), \dots, \sm(y_k)) = \sm(\phi(y_1, \dots, y_k)) \]
 is a red edge, and so $\phi( x_1, \dots, x_k) \sim \phi( y_1, \dots, y_k)$ is a blue edge. This is enough. 
\end{proof}

We will use this in Section \ref{sect:pr}, but for now we exhibit its utility by using it to characterise the graphs $\PH$ having conservative $\WNU$-polymorphisms. This 
yields a simple complexity dichotomy for the problem $\lpHom{H}$.

 \begin{theorem}\label{thm:main}
 For a $\br$-graph $H$, $\PH$ admits a conservative $\WNU$-polymorphism if and only if $\redSub{\PH}$ admits a conservative $\kNU{3}$-polymorphism.
\end{theorem}

The proof of this theorem takes the rest of the section, but consists mostly of observations about known results. We recall these results and then give the proof formally at the end of the section. 
First observe the following hierarchy of polymorphisms.
\begin{eqnarray}\label{eq:easydir}
 \PH \mbox{ has a conservative } \kNU{3} & \Rightarrow & \PH \mbox{ has a conservative } \WNU \\ \notag
 & \Rightarrow & \redSub{\PH} \mbox{ has a conservative } \WNU \\ \notag
 & \Rightarrow & \redSub{\PH} \mbox{ has a conservative } \kNU{3} 
\end{eqnarray}
The first two implications are trivial; indeed, the first is by definition, and the second is because a polymorphism of a structure is a polymorphism of any relation of the structure. We get the third implication by applying Theorem \ref{thm:FHHshort} to the graph $\redSub{\PH}$.
To prove Theorem \ref{thm:main} it is enough to show that a $\kNU{3}$-polymorphism on $\redSub{\PH}$ may be assumed to be switch-symmetric, as then by Lemma \ref{lem:ss} we get that the four statements in \eqref{eq:easydir} are equivalent. 
In fact one can assume more about the $\kNU{3}$-polymorphism on $\redSub{\PH}$.

 \begin{definition}\label{def:commute}
 A conservative $\kNU{3}$-polymorphism $\phi: H^3 \to H$ is {\em symmetric} if it commutes with all automorphisms $\sigma$ of $H$:
 \[ \phi(\sigma(a),\sigma(b),\sigma(c)) = \sigma(\phi(a,b,c)). \]
 \end{definition}

 
 In \cite{HR11}, while giving a characterisation of the digraphs $H$ for which $\lHom(H)$ is tractable, Hell and Rafiey found a useful omitted subgraph characterisation of digraphs with conservative $\NU$-polymorphisms. Their definition is for digraphs, but we give it here for graphs.

 \begin{definition}
 A walk $P = u_0\sim u_1\sim \dots \sim u_n$ of a graph $H$ {\em avoids} a walk $Q = v_0 \sim v_1 \sim \dots \sim v_n$ of $H$ of the same length if $u_{i-1} \not\sim v_{i}$ for each $i \in [n]$.
 For a $3$-tuple $(a,b,c)$ of vertices of $H$, a {\em $b$-excluder} in $H$ is a set of three walks $B_a, B_b$ and $B_c$ of the same length, starting at $a,b$ and $c$ respectively, such that
 $B_a$ and $B_c$ share the same last vertex, and $B_b$ avoids $B_a$ and $B_c$. See Figure \ref{fig:BExcluder}. 
 A $3$-tuple $(a,b,c)$ is a {\em permutable triple} if $H$ contains an $a$-excluder, a $b$-excluder, and a $c$-excluder.
 \end{definition}

 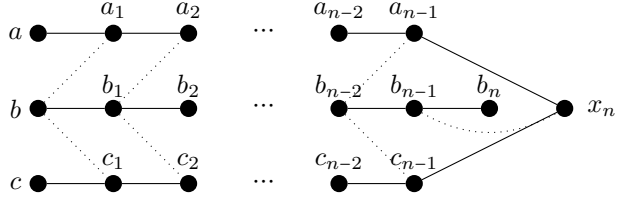
\begin{figure}
 \begin{center}
 \begin{tikzpicture}[every node/.style={blackvertex}]
 \Verts[]{a5/0/2, a4/1/2, a3/2/2, a2/4/2, a1/5/2, x0/7/1}
 \Verts[]{b5/0/1, b4/1/1, b3/2/1, b2/4/1, b1/5/1, b0/6/1}
 \Verts[]{c5/0/0, c4/1/0, c3/2/0, c2/4/0, c1/5/0}
 \Vlabels{a/left/a5, a_1/above/a4, a_2/above/a3, a_{n-2}/above/a2, a_{n-1}/above/a1}
 \Vlabels{b/left/b5, b_1/above/b4, b_2/above/b3, b_{n-2}/above/b2, b_{n-1}/above/b1, b_n/above/b0}
 \Vlabels{c/left/c5, c_1/above/c4, c_2/above/c3, c_{n-2}/above/c2, c_{n-1}/above/c1}
 \foreach \i in {0,1,2}{\draw (3,\i) node[empty](){$\cdots$};}

 \draw (7.5,1) node[empty](){$x_n$};
 \Edges[black]{a5/a4, a4/a3, a2/a1, a1/x0}{}
 \Edges[black]{b5/b4, b4/b3, b2/b1, b1/b0}{}
 \Edges[black]{c5/c4, c4/c3, c2/c1, c1/x0}{}
 \Edges[black,dotted]{b5/a4, b4/a3, b2/a1}{}
 \Edges[black,dotted]{b5/c4, b4/c3, b2/c1}{}
 \Edges[black,dotted, bend right]{b1/x0}{}
 \end{tikzpicture}
 \end{center}
 \caption{A $b$-excluder for $(a,b,c)$} \label{fig:BExcluder}
 \end{figure}

In Theorem 4.1 of \cite{HR11}, it was shown that a graph $H$ has a $\kNU{3}$-polymorphism if and only if it contains no permutable triples. In the proof, it was observed that 
if a graph $H$ has no permutable triples then for any $3$-tuple $(x_1,x_2,x_3)$ in $V(H)^3$, there is at least one $i \in [3]$ such that there is no $x_i$-excluder in $H$ for $(x_1,x_2,x_3)$.
We isolate a fact shown within the proof of Theorem 4.1 of \cite{HR11}.

 \begin{lemma}[\cite{HR11}]\label{lem:def3nu}
 Let $H$ be a graph containing no permutable triples.
 Let $\phi(x_1,x_2,x_3) = m$ if at least two of the entries are $m$, and otherwise
 let $\phi(x_1,x_2,x_3) = x_i$ where $i$ is the minimum index such that there is no $x_i$-excluder in $H$ for $(x_1,x_2,x_3)$.
 The function $\phi:H^3 \to H$ is a conservative $\kNU{3}$-polymorphism of $H$.
 \end{lemma}

 It is clear that the existence of excluders for a triple is preserved under any automorphism $s$ of $H$. The properties `two of $x_1, x_2$ and $x_3$ being the same', and `minimum index' are also preserved, so the conservative $\kNU{3}$-polymorphism defined in the lemma is symmetric.
 This allows us to prove Theorem \ref{thm:main}.
 \begin{proof}[Proof of Theorem \ref{thm:main}]
 Let $\redSub{\PH}$ have a conservative $\kNU{3}$-polymorphism $\phi$. By the discussion following Lemma \ref{lem:def3nu} we may assume that $\phi$ is symmetric, and so, as the switch map is an automorphism of $\redSub{\PH}$,
 switch-symmetric. By Lemma \ref{lem:ss}, $\phi$ is then a conservative $\kNU{3}$-polymorphism of $\PH$, and so a conservative $\WNU$-polymorphism.
 This completes the proof of the equivalence of the four polymorphisms in \eqref{eq:easydir}, and so proves Theorem \ref{thm:main}.
\end{proof}

 \section{Reduction to bipartite graphs and bipartite-min orderings} \label{sect:bipart}
 
 In \cite{FHH}, Feder, Hell, and Huang, reduced the complexity of the list homomorphism problem for graphs to the bipartite (so irreflexive) case by showing that $\lHom(H)$ for any graph $H$ is polynomially equivalent to $\lHom(\Bip{H})$ for a bipartite graph $\Bip{H}:= K_2 \times H$. Their result implies that $H$ has a conservative $\WNU$-polymorphism if and only if $\Bip{H}$ does. 
 Their proof consists of two parts. Where `parity-symmetric' polymorphisms on $\Bip{H}$ are defined below, the authors of \cite{FHH} essentially prove the following two equivalences. 
 
 \begin{enumerate}
 \item $H$ has conservative $\WNU$-polymorphism if and only if $\Bip{H}$ has a parity-symmetric conservative $\WNU$-polymorphism. 
 \item $\Bip{H}$ has a parity-symmetric conservative $\WNU$-polymorphism if and only if it has a conservative $\WNU$-polymorphism. 
 \end{enumerate}
 
 (To reconcile these statements with those of \cite{FHH}, the reader should observe that the graph $H^{**}$ of \cite{FHH} is simply $\Bip{H}$ augmented with a relation that forces all polymorphisms to be parity-symmetric.) 
 
 The first statement is easy. In this section we observe that it generalises and extends to similar statements about polymorphisms of $\br$-graphs, and in particular about polymorphisms of $\PH$. The second statement is much harder and depends on structural results.


 \subsection{The bipartite resolution and the constant parity subgraph} 
 
 Recall that for a graph $H$, the categorical product $K_2 \times H$ is the graph with vertex set $\{0,1\} \times V(H)$ in which $(i,h) \sim (i',h')$ if and only if $i \neq i'$ and $h \sim h'$. 
 
 Let $B$ be the purple $K_2$ on the vertex set $\{0,1\}$. For a $\br$-graph $H$, the {\em bipartite resolution} of $H$ is $\Bip{H}:= B \times H$. It is a bipartite $\br$-graph with partite sets $\pi_B^{-1}(0)$ and $\pi_{B}^{-1}(1)$,
 so we call $i = \pi_{B}((i,v))$ the {\em parity} of a vertex $(i,v)$ and call the map
 \begin{equation}\label{eq:pfmap}
\fp : (i,v) \mapsto (1-i,v), 
 \end{equation} 
 which is clearly an automorphism of $\Bip{H}$, the {\em parity flip} map. A polymorphism $\phi$ of $\Bip{H}$ is {\em parity-symmetric} if it commutes with the parity flip map. 
 
 Let $\cP$ be the set consisting of the following properties applying to polymorphisms of a graph $H$: 
 \begin{quote}
 idempotent, conservative, semi-conservative (if $H$ is a switch graph), $\kNU{k}$, $\WNU$.
 \end{quote}

 For any subset $P \subseteq \cP$, a polymorphism $\phi$ of $H$ is a $P$-polymorphism if it has all properties in $P$. The main result of this section is the following. 
 
 \begin{lemma}\label{lem:bipart}
 For any $\br$-graph $H$, and any subset $P$ of $\cP$, $H$ has a $P$-polymorphism $\phi$ if and only if $\Bip{H}$ has a parity-symmetric $P$-polymorphism $\Phi$ (called the bipartite resolution of $\phi$). 
 \end{lemma}
 \begin{proof}
 This is immediate from Claims \ref{cl:bip-ext} and \ref{cl:ext} given below.
 \end{proof}
 
 For a vertex $v = (v_1, \dots, v_k)$ in $\Bip{H}^k$, the {\em parity pattern} of $v$ is
 the tuple
 \[ \pi_B(v):= ( \pi_B(v_1), \dots \pi_B(v_k)) \]
 of its coordinates' parities. 
 Over components of $\Bip{H}$ the parity pattern is constant up to component-wise action of the parity flip map $\fp$. In particular the {\em constant parity subgraph} $\CPC{k}$ of $\Bip{H}^k$, which is the subgraph induced on the vertices with parity patterns $(0,0, \dots, 0)$ or $(1,1, \dots, 1)$, is a union of components. A coordinate of a parity pattern is {\em majority} if at least half of the coordinates have the same parity. 
 A homomorphism $\Phi: \CPC{k} \to \Bip{H}$ is a $P$-homomorphism, for any $P \subseteq \cP$, if it satisfies those same conditions required of a polymorphism of $\Bip{H}$ to satisfy them, and is parity-symmetric if it commutes with $\fp$.

 \begin{claim}\label{cl:bip-ext} 
 For any $\br$-graph $H$ and any $P \subseteq \cP$, $H$ has a $k$-ary $P$-polymorphism if and only if $\CPC{k}$ has a parity-symmetric $P$-polymorphism. 
 \end{claim}
 \begin{proof}
 The details of the proof are straightforward but tedious, and so we just outline the main points. 
 
 For a polymorphism $\phi: H^k \to H$, the map $\Phi: \CPC{k} \to \Bip{H}$ defined for $v = (v_1, \dots, v_k) \in V(\CPC{k})$ by 
 \[ \Phi(v) = (\pi_B(v_1), \phi(\pi_H(v_1), \dots \pi_H(v_k)) ) \]
 is a parity-symmetric homomorphism, and for a parity-symmetric homomorphism 
 $\Phi: \CPC{k} \to \Bip{H}$ the map $\phi:H^k \to H$ defined for $v = (v_1, \dots, v_k)$ by 
 \[ \phi( v ) = \Phi( (0,v_1), \dots, (0,v_k) ) \]
 is a polymorphism; moreover, these constructions are inverse. 
 
 That $\phi$ is $\WNU$ or $\kNU{k}$ if and only if $\Phi$ is also straightforward. 
 To verify that $\phi$ is idempotent, conservative, or semi-conservative if and only if $\Phi$ is one need only observe that $\Phi$ preserves a list $L \subset V(\Bip{H})$ if and only if $\phi$ preserves its intersections,  $\pi_H( L \cap \pi_B^{-1}(0))$ and $\pi_H( L \cap \pi_B^{-1}(1))$,    with each side of $\Bip{H}$.
 \end{proof}
 
 \begin{claim}\label{cl:ext} 
 For any $\br$-graph $H$ and any $P \subseteq \cP \cup \{ \mbox{parity-symmetric}\}$, any $P$-homomorphism $\CPC{k} \to \Bip{H}$ extends to a $P$-polymorphism of $\Bip{H}$.
 \end{claim}
 \begin{proof}
 Given a homomorphism $\Phi: \CPC{k} \to \Bip{H}$, extend it to a polymorphism of $\Bip{H}$ by defining it, on each component of $\CPC{k}$, to project onto some coordinate that is majority in the parity pattern of the component.
 
 That this maintains the properties of idempotence, conservativity, and semi-conservativity is immediate from the observation that projections preserve all lists. That parity-symmetry is maintained is immediate from the fact that projections are parity-symmetric. That the properties $\kNU{k}$ and $\WNU$ are maintained uses the choice of projection, but it straightforward. 
 \end{proof}
 
 Applying Lemma \ref{lem:bipart} to the $\br$-graph $\cS_H$ yields this useful observation. 
 \begin{proposition}\label{prop:bipWNU} 
   For $\br$-graph $H$, $\cS_H$ has a semi-conservative $\WNU$-polymorphism if and only if 
   $\Bip{\cS_H}$ has a parity-symmetric semi-conservative $\WNU$-polymorphism.  
 \end{proposition}
 
 We use this in conjunction with one more simple observation.
 
 \begin{fact}\label{fact:pbh}
 For any $\br$-graph $H$, $\Bip{\PH}$ and $\cS_{\Bip{H}}$ are isomorphic.
 \end{fact}
 \begin{proof}
 They are both the graph with vertex set $B \times V(H) \times S$ where 
 $(i,h,j) \sim (i',h',j')$ if and only if $i \neq i'$ and $h \sim h'$.
 The edge $(i,h,j)(i',h',j')$ is red if and only if $hh'$ is red and $j = j'$ or $hh' $ is blue and $j \neq j'$.
 \end{proof}
 
 \subsection{bipartite-min orderings}
 As mentioned in the introduction, a graph cannot have a min ordering unless almost all of its vertices have loops.  In \cite{FV98}bipartite-min orderings were considered for bipartite graphs.  
 \begin{definition}
 An ordering of the vertices of a bipartite graph $H$ with partite sets $U$ and $V$ is a {\em bipartite-min ordering} if for any $a,b \in U$ and $a',b' \in V$, the underbar property \eqref{eq:underbar} holds. 
 \end{definition}
 
 As a min-ordering $\leq$ of $H$ is a $2$-ary function $\min$ which defines a conservative $3$-ary $\WNU$-polymorphism $\min$ of $H$, (see before and after Definition \ref{def:min}) a bipartite-min ordering of $H$ is two functions, one on the partite set $U$ and one on $V$, that defines a conservative  $\WNU$-homomorphism from $\CPC{3}$ to $\Bip{H}$, so by Claim \ref{cl:ext} extends to a conservative  $\WNU$-polymorphism of $\Bip{H}$.  
 So like a min-ordering of $H$ yields a conservative $\WNU$-polymorphism on $H$, a bipartite-min ordering of $\Bip{H}$ yields a conservative $\WNU$-polymorphism on $H$. If this polymorphism of $\Bip{H}$ is parity-symmetric, if gives, by Lemma \ref{lem:bipart}, a conservative $\WNU$-polymorphism on $H$. 
   
 Viewing min orderings and bipartite-min orderings as functions, we get the following case of Lemma \ref{lem:bipart}.   
 
 \begin{fact}\label{fact:minbipmin}
   A graph $H$ has a min ordering if and only if $\Bip{H}$ has a parity-symmetric bipartite-min polymorphism.
 \end{fact}

 \section{Bipartite-min orderings and purple-red graphs}\label{sect:pr}

 As useful as min orderings are, one looks to define them for $\br$-graphs.  This was done for trees and some other classes of $\br$-graphs in \cite{BBFHJ}. The authors defined special min orderings of reflexive $\br$-trees, and special bipartite-min orderings of irreflexive $\br$-trees, and then showed that for $\br$-graphs $H$ with these orderings the problem $\lsHom{H}$ can be solved in polynomial time.  
 
 The polynomial times algorithms from \cite{BBFHJ} are complicated.  In this section we suggest an approach to these tractability proofs that is easier, but depends on the CSP-dichotomy.  
 
We consider $\pr$-graphs, and we will assume that they have been switched to remove all pure blue edges. We thus talk of red edges, by which we mean pure red edges, and purple edges. In the case that we want to consider a purple edge as red, say in the target of a homomorphism, we will say `red or purple' or `not necessarily pure red'. 

 The following definition is from \cite{BBFHJ} but simplified for $\pr$-graphs. 
 
\begin{definition}\label{def:special}
 Given a $\pr$-graph $H$ with a (bipartite-)min ordering of the underlying graph, a vertex is {\em special} if all its neighbours via purple edges come before  its other neighbours. The ordering is a {\em special (bipartite-)min ordering} if all vertices are special. 
\end{definition}

 Our main theorem says that the existence of a special bipartite-min ordering on $\Bip{H}$, for a reflexive $\pr$-graph $H$, yields a semi-conservative $\WNU$ on $\cS_H$, so puts $\lsHom(H)$ in $P$.

 A {\em red component} $R$ of a $\pr$-graph $H$ is the subgraph induced by a set of vertices that are connected with red edges. We note that a red component may induce purple edges, but is connected via red edges. 
 A neighbour of a vertex is called {\em purple} or {\em red} if it is adjacent via a purple or red edge.

 \begin{lemma}\label{lemma:rededgeup}
 Let $H$ be a bipartite $\pr$-graph having a special bipartite-min ordering, and let $R$ be a red component of $H$ with maximum vertices $u_a$ and $u_b$. Let $r_ar_b$ be a red edge in $R$, and let $x_ax_b$ be an edge with 
 $r_a \leq x_a \leq u_a$ and $r_b \leq x_b$. Then $x_ax_b$ is red, and $x_a$ and $x_b$ are in $R$. 
 \end{lemma}
 \begin{proof}
 First we prove the lemma under the assumption that $x_b \leq u_b$. 
 As $R$ is a red component there is a red path from $r_a$ to $u_a$, and somewhere it must cross $x_ax_b$ so we may assume (by switching the roles of $a$ and $b$ if necessary) we have a red path $aba'$ with $a \leq x_a < a'$ and $b \leq x_b$. By the underbar property we have $x_a \sim b$.
 As $b$ is special $x_ab$ must be red, and as $x_a$ is special $x_ax_b$ must be red, and in $R$. 
 
 Now assume that $u_b < x_b$. If $x_a$ is in $R$ then it has some red neighbour below $x_b$ and so $x_ax_b$ is red by specialty of $x_a$. 
 If $x_a$ is not in $R$ there is a red edge $u_ay_b$ with $y_b \leq x_b$ which says that $x_a \sim y_b$ by the underbar property.
 So $x_a$ is in $R$ by the previous case of the lemma. 
 \end{proof}
 
 This shows that if there is a purple edge $x_ax_b$ above a red edge $r_ar_b$, then it must be above the whole red component $R$ containing $r_a$ and $r_b$. 
 
 

 

 We are now ready to prove the main result of the section. 
 
 \begin{theorem}\label{thm:specbmo} 
 Let $\Bip{H}$ be a bipartite $\pr$-graph with a special bipartite-min ordering $\leq$. The switch graph $\cS:=\cS_{\Bip{H}}$ has a 
 semi-conservative $\WNU$-polymorphism $\Phi$.
 Moreover, if $\leq$ is parity-symmetric, then so is $\Phi$. 
 \end{theorem}
 \begin{proof}
 For a vertex $v$ of $\Bip{H}$ let $\rc(v)$ be the {\em red component} containing $v$.
 By Claim \ref{cl:bip-ext} it is enough to define $\Phi: {\cS}^3 \to \cS$ on constant parity tuples. For such tuples we set
 \[ \Phi( (a_1,x_1), (a_2,x_2), (a_3,x_3) ) = ( A, X )\]
 where $A = \min(a_1,a_2,a_3)$ for the bipartite-min polymorphism $\min$ defined by $\leq$, and $X$ is defined as follows.
 Calling $(a_i,x_i)$ or simply $a_i$ {\em relevant} if $\rc( a_i ) = \rc(A)$, let
 \begin{itemize}
 \item $X = \maj(x_1,x_2,x_3)$ if all $a_i$ are relevant; otherwise let
 \item $X = x_i$ where $i$ is the minimum $i$ such that $a_i$ is relevant. 
 \end{itemize}

 The function $\Phi$ is clearly idempotent, and semi-conservative. To see that it is $\WNU$, it is enough to show that $A$ and $X$ are both $\WNU$. Certainly $A$ is, to see that $X$ is too, we assume that there is a repeated value in $\{a_1, a_2, a_3\}$ and observe that the value of $X$ is chosen from among $\{x_1,x_2,x_3\}$ depending on which of the $a_i$ is repeated, and is unchanged under permuting the indices $i = 1,2,3$. Indeed, if the repeated $a_i$ is not relevant then $X$ is the other $x_i$.
 If the repeated $a_i$ is relevant, while the other is not, $X$ is the repeated $x_i$. If all are relevant then $X$ is majority, and so is the repeated $x_i$. In all cases, $X$ is $\WNU$.
 Observe also that $\Phi$ is switch-symmetric, as complementing the $x_i$ does not effect $A$, and so does not effect which $x_i$ is returned by $X$. 
 
 What is left to be shown, for the first statement of the theorem, is that $\Phi$ is a homomorphism. As it is switch-symmetric, by Lemma \ref{lem:ss} it is enough to show that it preserves (not necessarily pure) red edges. 
 Assume that $(a_i,x_i) \sim (b_i,y_i)$ is a red (or purple) edge for $i=1,2,3$. 
 We show that 
 \[ (A,X) := \Phi( (a_1,x_1), (a_2,x_2), (a_3, x_3) ) \sim \Phi( (b_1,y_1), (b_2,y_2), (b_3, y_3) ) =: (B,Y), \]
 is a red edge (or purple) edge.

 Indeed, as $\min$ is a polymorphism of the underlying graph $\Bip{H}^+\iso \redSub{\Bip{H}}$ of $\Bip{H}$, we have
 \[ A = \min(a_1,a_2,a_3) \sim \min(b_1,b_2,b_3) = B\]
 in $\redSub{\Bip{H}}$, so $AB$ is a red (or purple) edge of $\Bip{H}$. If it is purple then whatever $X$ and $Y$ are, we have that $(A,X)(B,Y)$ is a red (or purple) edge, and we are done. So we must show that if $AB$ is pure red then $X = Y$. 
 
Assume that $AB$ is pure red. First observe, for all $i$, that if $a_i$ is relevant to $A$, then $a_ib_i$ is pure red, making $x_i = y_i$, and $b_i$ is relevant to $B$. 
 Indeed, if $a_i$ is relevant we have $A \leq a_i$ and $a_i$ is in $\rc(A)$. Any purple neighbour of $a_i$ must be below $B$ by the fact that $a_i$ is special and by Lemma \ref{lemma:rededgeup}, so $a_ib_i$ must be a pure red edge, making $b_i$ relevant to $B$. Similarly $a_i$ is relevant if $b_i$ is, and this holds for all relevant arguments, so with the fact that $x_i = y_i$ for these arguments, we get $X = Y$.

 For the second statement of the theorem, we have that $A$ is parity-symmetric by definition, as $\leq$ is parity-symmetric, so it is enough to observe that $X$ is unchanged by replacing the $a_i$ with $\ps(a_i)$. As $A$ is parity-symmetric, doing so does not change which $a_i$ are relevant, so this is immediate from the definition of $X$.
 \end{proof}
 
 We finish off this section by showing how our theorem yields Conjecture \ref{conj:main2} for reflexive $\pr$-graphs. 
 \begin{theorem}\label{thm:Main}
 For a reflexive $\pr$-graph $H$, the following are equivalent.
 \begin{enumerate}
 \item $H$ has no invertible pairs or chains.
 \item[(ii*)] $\Bip{H}$ has a parity-symmetric special bipartite-min ordering.
 \item[(iii)] $\PH$ has a semi-conservative $\WNU$-polymorphism. 
 \end{enumerate}
 \end{theorem}
 \begin{proof}
 In \cite{BBHJR} it is shown that if $H$ have no invertible pairs or chains, then $H$ has a special min ordering.  By Fact  \ref{fact:minbipmin}, a min ordering of $H$ is exactly a bipartite-min ordering of $\Bip{H}$, and by definition, the first is special if and only if the second is, so we have the implication $(i) \Rightarrow (ii^*)$.   
 
 If $\Bip{H}$ has a parity-symmetric special bipartite-min ordering,  Theorem \ref{thm:specbmo} gives 
 a parity-symmetric semi-conservative $\WNU$-polymorphism of $\cS_{\Bip{H}}$.
 This is $\Bip{\cS_H}$ by Fact \ref{fact:pbh}, and so by Proposition \ref{prop:bipWNU}, $\cS_H$ has a semi-conservative $\WNU$-polymorphism. 
 This gives us $(ii^*) \Rightarrow (iii)$.
 
 Implication $(iii) \Rightarrow (i)$ also uses known results.  Assume that 
 $\PH$ has a has a semi-conservative $\WNU$-polymorphism. Then the problem $\lsHom(H)$ is in $\P$ by Fact \ref{fact:WNU}, and in particular it is in $\P$ when restricting to red instances.
 As $H$ the underlying graph of $H$ is exactly the red subgraph $\redSub{H}$, switching does not help for red instances, and so $\lsHom(H)$ for red instances is just $\lHom(\redSub{H})$, putting $\lHom(\redSub{H})$ in $\P$. Thus $\redSub{H}$ has no invertible pairs by Theorem \ref{thm:min-order-irr}. By definition, an invertible pair of $H$ is an invertible pair of the underlying graph, which in this case is $\redSub{H}$, and so $H$ has no invertible pairs. 
 In \cite{BBFHJ} it is shown that chains in $H$ omit semi-conservative $\WNU$-polymorphisms on $\PH$, so it has no chains either. 
 \end{proof}

 \section{Possible and impossible polymorphism collapses}\label{sect:counter}


 The following questions arise naturally when considering the complexity dichotomy for $\lsHom(H)$, and when observing the various polymorphism collapses we have mentioned in the paper. 
 
\begin{question}\label{questions}
Is it true for a $\br$-graph $H$ that 
\begin{enumerate}
 \item $\PH$ has a semi-conservative $\WNU$ if and only if it has a parity-symmetric one?
 \item $\PH$ has a semi-conservative $\WNU$-polymorphism if and only if it has a conservative $\WNU$-polymorphism?
 \item $\PH$ has a semi-conservative $\WNU$-polymorphism if and only if it has a semi-conservative $\kNU{3}$-polymorphism?
\end{enumerate}
\end{question}

Part (i) is certainly tempting. Using the results and ideas of Section \ref{sect:bipart} it is the key to getting a full reduction to the bipartite case:
 $\lsHom(H) \peq \lsHom(\Bip{H})$. 
Part (ii) was tempting, reducing the whole paper to Section \ref{sect:br}, and (iii) also, as it would reduce the paper to Sections \ref{sect:br} and \ref{sect:bipart}. We now have countless counter-examples giving negative answers to these last two questions. We give one in the proposition below.

 \begin{figure}
 \begin{center}
 \begin{tikzpicture}[every node/.style={blackvertex}, scale=.6]
 \Verts[]{a0/0/0, a1/0/1, a2/0/2, a3/0/3, a4/0/4, b4/-1/4.5, b5/-1/5.5}
 \Edges[red, thick, dashed]{a0/a1, a1/a2, a2/a3, a3/a4}{bend left = 20}
 \Edges[blue, thick ]{a0/a1, a1/a2, a2/a3, a3/a4}{bend right = 20}
 \Edges[red, thick, dashed]{a3/b4,b4/b5}{}

 \begin{scope}[xshift = 5cm]
 
 \Verts[]{a0/0/0, a1/0/1, a2/0/2, a3/-1/3, a3'/1/3, a4/0/3, b4/-1/4.5, b5/-1/5.5, b4'/1/4.5, b5'/1/5.5}
 \Edges[red, thick, dashed]{a0/a1, a1/a2, a2/a3, a3/a4, a3'/a2, a4/a3'}{bend left = 20}
 \Edges[blue, thick ]{a0/a1, a1/a2, a2/a3, a3/a4, a3'/a2, a4/a3'}{bend right = 20}
 \Edges[red, thick, dashed]{a3/b4,b4/b5,a3'/b4',b4'/b5'}{}
 \Edges[blue, thick]{a3/b4',b4/b5',a3'/b4,b4'/b5}{}
 \end{scope}

 \begin{scope}[xshift = 10cm]
 
 \Verts[]{a0/0/0, a1/0/1, a2/0/2, a3/-1/3, a3'/1/3, a4/0/4, b4/-1/4.5, b5/-1/5.5, b4'/1/4.5, b5'/1/5.5}
 \Edges[red, thick, dashed]{a0/a1, a1/a2, a2/a3, a3/a4, a3'/a2, a4/a3'}{}
 \Edges[red, thick, dashed]{a3/b4,b4/b5,a3'/b4',b4'/b5'}{}
 \Vlabels{a_2/left/a0, a_1/left/a1, a_0/left/a2, b_0/left/a3, b_1/left/b4, b_2/left/b5, c_0/right/a3', c_1/right/b4', c_2/right/b5', d/above/a4} 
 \end{scope}
 \end{tikzpicture}
 \end{center}
 \caption{The tree $T$, the switch-core of $\cP_{T}$, and its red subgraph $R$}\label{fig:counter}
 \end{figure}
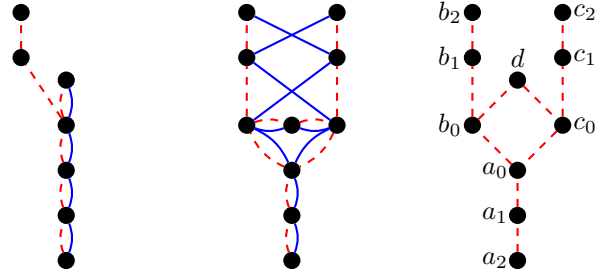


\begin{proposition}
 Where $T$ is the $\br$-graph on the left of Figure \ref{fig:counter}, $\cP_T$ has a semi-conservative $\WNU$-polymorphism but no conservative $\WNU$-polymorphism or semi-conservative $\kNU{3}$-polymorphism.
 \end{proposition}
 \begin{proof}
 That $\cP_T$ has a semi-conservative $\WNU$-polymorphism is shown in \cite{BBFHJ}, and also follows from Theorem \ref{thm:specbmo}. Indeed, ordering the vertices of $T$ according to their position in the figure as one moves up the
 page; except putting the top vertex in the purple path below its neighbour, we have a special bipartite-min ordering of $T$. 

 It is not too hard to verify that the triple $(a_2,b_1,c_1)$ in $R$ is a permutable triple, showing that it and so $\redSub{\cP_{T}}$ (of which it is an induced subgraph) can have no conservative $\WNU$-polymorphism.
 
 We show now that $R$ can have no semi-conservative $\kNU{3}$-polymorphisms. Note that now symmetric lists are those for which if $b_i$ is in a list then $c_i$ is also, and vice-versa. 
 
Assume, towards contradiction, that $\phi$ is a semi-conservative $\kNU{3}$-polymorphism of $R$. There is no valid image for $\phi(a_2,b_1,c_1)$. Indeed, if $\phi(a_2, b_1, c_1) = a_2$, then $\phi(a_1, b_0, c_0) = a_1$, and $\phi(a_2,d,d) = a_2$. This contradicts the fact that $\phi$ is a $\kNU{3}$. So without loss of generality we may assume that $\phi(a_2, b_1, c_1) = b_1$, Then $\phi(a_1,b_2,c_2) = b_2$, $\phi(a_0,b_1,c_1) = b_1$, $\phi(c_0,b_2,c_2) \in \{b_0, b_2\}$ and $\phi(c_1,b_1,c_1) = b_1$, again contradicting the fact that $\phi$ is a $\kNU{3}$. 
\end{proof}

 In \cite{BBFHJ} the authors observe that there are infinitely many trees having semi-conservative $\WNU$-operations that contain $T$ as an induced subgraph. The proof works for all of these.

 \section*{Acknowledgements}
 
 We thank the authors of \cite{BBFHJ} for early copies of their manuscript, and Pavol Hell in particular for generously directing us towards key ideas on several occasions. We thank an anonymous reviewer for a very detailed reading which revealed essential mistakes in an earlier version of the paper.

 \providecommand{\bysame}{\leavevmode\hboxto3em{\hrulefill}\thinspace}
 \newcommand{\doi}[1]{\href{http://dx.doi.org/#1}{\small\nolinkurl{DOI: #1}}}

\end{document}